\documentclass[12pt]{amsart}
\usepackage{amsfonts,amssymb,amscd,amsmath,enumerate,verbatim,url}
\usepackage[noxcolor]{pstricks}
\usepackage[dvips]{graphicx}
\def\NZQ{\mathbb}               

\def\ZZ{{\NZQ Z}}
\def\RR{{\NZQ R}}

\def\PP{{\NZQ P}}

\def\ab{{\bold a}}
\def\bb{{\bold b}}

\def\eb{{\bold e}}
\def\opn#1#2{\def#1{\operatorname{#2}}} 
\opn\cone{cone}
\opn\aff{aff} \opn\con{conv} \opn\relint{relint} \opn\st{st}
\opn\lk{lk} \opn\cn{cone} \opn\core{core} \opn\vol{vol}
\opn\link{link} \opn\star{star}
\opn\gr{gr}

\def\Fc{{\mathcal F}}

\newtheorem{Theorem}{Theorem}[section]
\newtheorem{Lemma}[Theorem]{Lemma}
\newtheorem{Corollary}[Theorem]{Corollary}
\newtheorem{Proposition}[Theorem]{Proposition}
\newtheorem{Remark}[Theorem]{Remark}

\newtheorem{Definition}[Theorem]{Definition}

\opn\dis{dis}
\opn\height{height}
\opn\dist{dist}
\def\pnt{{\raise0.5mm\hbox{\large\bf.}}}

\opn\Lex{Lex}

\begin{document}
\title{Equivalence classes for smooth Fano polytopes}
\author{Akihiro Higashitani}
\address{Akihiro Higashitani,
Department of Mathematics, Kyoto University, Japan 
Kitashirakawa Oiwake-cho, Sakyo-ku, Kyoto, 606-8502, Japan}
\email{ahigashi@math.kyoto-u.ac.jp}
\begin{abstract}
Let $\Fc(n)$ be the set of smooth Fano $n$-polytopes up to unimocular equivalence. 
In this paper, we consider the F-equivalence or I-equivalence classes for $\Fc(n)$ 
and introduce F-isolated or I-isolated smooth Fano $n$-polytopes. 
First, we describe all of F-equivalence classes and I-equivalence classes for $\Fc(5)$. 
We also give a complete characterization of F-equivalence classes (I-equivalence classes) 
for smooth Fano $n$-polytopes with $n+3$ vertices and construct a family of I-isolated smooth Fano polytopes. 
\end{abstract}
\maketitle

\section{Introduction}
Let $P \subset \RR^n$ be a lattice $n$-polytope, i.e., 
a convex polytope of dimension $n$ whose vertices belong to $\ZZ^n$. 
Let $V(P)$ denote the set of the vertices of $P$. We recall some notions on lattice polytopes. 
\begin{itemize}
\item We say that $P$ is {\em reflexive} if $P$ contains the origin in its interior and its dual polytope 
$P^\vee = \{ y \in \RR^n : \langle x, y \rangle \leq 1 \text{ for every } x \in P\}$ is also a lattice polytope, 
where $\langle \cdot, \cdot \rangle$ stands for the usual inner product on $\RR^n$. 
\item We say that $P$ is {\em simplicial} if each facet of $P$ contains exactly $n$ vertices. 
\item We say that $P$ is a {\em smooth Fano polytope} if the origin is contained in its interior and 
the vertex set of each facet of $P$ forms a $\ZZ$-basis for $\ZZ^n$. 
In particular, every smooth Fano polytope is reflexive and simplicial. 
\item Two lattice polytopes $Q \subset \RR^n$ and $Q' \subset \RR^n$ are {\em unimodularly equivalent} 
if there is an affine map $\phi : \RR^n \rightarrow \RR^n$ such that $\phi(\ZZ^n)=\ZZ^n$ and $\phi(Q)=Q'$. 
\end{itemize}
As is well known, each smooth Fano $n$-polytope (up to unimodular equivalence) 
one-to-one corresponds to a toric Fano $n$-fold (up to isomorphism). 
Thus, knowing smooth Fano polytopes is equivalent to knowning toric Fano manifolds in some sense.

By many researchers, smooth Fano polytopes or toric Fano manifolds have been investigated. 
Especially, their classification for each dimension is one of the most interesting problem. 
The complete classification of smooth Fano polytopes or toric Fano manifolds 
was given by \cite{Bat81} and \cite{WW} in dimension 3, 
by \cite{Bat99} and \cite{Sato00} in dimension 4 and by \cite{KN} in dimension 5. 
Recently, an explicit algorithm classifying all smooth Fano polytopes 
has been constructed by M. {\O}bro in 2007 (see \cite{Obro07} and \cite{Obrophd}). 
The following table shows the number of unimodular equivalence classes for smooth Fano $n$-polytopes. 

\begin{table}[htb]
\begin{tabular}{|c|c|c|c|c|c|c|c|}
\hline
\text{$n$}                       &1 &2 &3  &4   &5   &6    &7      \\ \hline
\text{smooth Fano $n$-polytopes} &1 &5 &18 &124 &866 &7622 &72256  \\ \hline
\end{tabular}

\smallskip

\caption{The number of unimodular equivalence classes of smooth Fano $n$-polytopes with $n \leq 7$}
\end{table}

Let $\Fc(n)$ be the set of all unimodular equivalence classes for smooth Fano $n$-polytopes. 
The main concern of this paper is some equivalence classes for $\Fc(n)$ 
with respet to the following two equivalence relations:

\begin{Definition}[{\cite[Definition 1.1, 6.1]{Sato00}, \cite[Definition 1.1]{Obro08}}]\label{F-equiv}
We say that two smooth Fano $n$-polytopes $P$ and $Q$ are {\em F-equivalent} if there exists a sequence 
$P_0$, $P_1, \ldots,P_{k-1}$, $P_k$ of smooth Fano $n$-polytopes satisfying the following three conditions: 
\begin{itemize}
\item[(a)] $P$ and $Q$ are unimodularly equivalent to $P_0$ and $P_k$, respectively; 
\item[(b)] For each $1 \leq i \leq k$, we have $V(P_i) = V(P_{i-1}) \cup \{w\}$ with $w \notin V(P_{i-1})$ 
or $V(P_{i-1}) = V(P_i) \cup \{w\}$ with $w \not\in V(P_i)$ for some lattice point $w$; 
\item[(c)] If $w \in V(P_i) \setminus V(P_{i-1})$, then there exists a proper face $F$ of $P_{i-1}$ 
such that $w = \sum_{v \in V(F)}v$ and the set of facets of $P_i$ containing $w$ is equal to 
$$\{\con(\{w\}\cup(V(F') \setminus \{v\})) : F' \text{ is a facet of }P_{i-1}, F \subset F', v \in V(F)\}.$$ 
In other words, $P_i$ is obtained by taking a stellar subdivision of $P_{i-1}$ with $w$. 
If $w \in V(P_{i-1}) \setminus V(P_i)$, then the similar condition holds. 
\end{itemize}
\end{Definition}

\begin{Definition}[{\cite[Section 1]{Obro08}}]\label{I-equiv}
Two smooth Fano $n$-polytopes $P$ and $Q$ are called {\em I-equivalent} 
if there exists a sequence of smooth Fano $n$-polytopes satisfying 
the conditions {\em (a)} and {\em (b)} in Definition \ref{I-equiv}, 
i.e., {\em (c)} is not necessarily satisfied. 
\end{Definition}

Clearly, if two smooth Fano polytopes $P$ and $Q$ are F-equivalent, then those are also I-equivalent. 

\begin{Remark}\label{rima-ku}{\em 
(a) These definitions are also available to {\em complete nonsingular fans} by exchanging some terminologies 
with the ones for fans. For example, we may exchange the set of the vertices of a polytope 
into the set of the primitive vectors of 1-dimensional cones in a fan. 

(b) The condition (c) described in Definition \ref{F-equiv} is interpreted as the condition 
that two corresponding toric Fano manifolds $X(P_{i-1})$ and $X(P_i)$ 
are related with some equivariant blow-up or equivariant blow-down. 
On the other hand, I-equivalence does not necessarily correspond to an equilvariant blow-up (blow-down). 
}\end{Remark}

Let $\eb_1,\ldots,\eb_n$ be the unit coordinate vectors of $\RR^n$ and let 
$$T^n = \con(\{\eb_1,\ldots,\eb_n, -(\eb_1+\cdots+\eb_n)\}).$$ 
Then $T^n$ is the smooth Fano polytope $n$-polytope corresponding to the $n$-dimensional projective space $\PP^n$. 
Note that this is the unique smooth Fano $n$-polytope with $n+1$ vertices (up to unimodular equivalence). 
For a given positive integer $k$, let 
\begin{align*}
&V^{2k}=\con(\{\pm \eb_1,\ldots, \pm \eb_{2k}, \pm(\eb_1+ \cdots +\eb_{2k})\}), \\
&\widetilde{V}^{2k}=\con(\{\pm \eb_1,\ldots, \pm \eb_{2k}, \eb_1+ \cdots +\eb_{2k})\}. 
\end{align*}
Then it is well known that these are smooth Fano $2k$-polytopes. 
Note that $V^{2k}$ and $\widetilde{V}^{2k}$ are I-equivalent and each of them is also I-equivalent to $T^{2k}$, 
while $V^{2k}$ and $\widetilde{V}^{2k}$ are NOT F-equivalent when $k \geq 2$. 

We say that a polytope $P$ is {\em pseudo-symmetric} if $P$ contains a facet $F$ such that 
$-F$ is also a facet of $P$. By Ewald \cite{Ewald}, it is proved that 
every pseudo-symmetric smooth Fano polytope is unimodularly equivalent to 
$$V^{2k_1} \oplus \cdots \oplus V^{2k_p} \oplus \widetilde{V}^{2l_1} \oplus \cdots \oplus \widetilde{V}^{2l_q} 
\oplus T_1 \oplus \cdots \oplus T_1,$$ where for two reflexive polytopes $P \subset \RR^d$ and $Q \subset \RR^e$, 
$P \oplus Q \subset \RR^{d+e}$ denotes the {\em free sum} of $P$ and $Q$, i.e., 
$$P \oplus Q = \con(\{(\alpha, 0) \in \RR^{d+e} : \alpha \in P\} \cup \{(0,\beta) \in \RR^{d+e} : \beta \in Q\}).$$
Note that the free sum of smooth Fano polytopes corresponds to the direct product of toric Fano manifolds. 

\bigskip

Sato \cite{Sato00} investigated the F-equivalence classes for $\Fc(n)$ as follows: 
For every $P \in \Fc(2)$ (resp. $P \in \Fc(3)$), $P$ is F-equivalent to $T^2$ (resp. $T^3$), 
namely, each of $\Fc(2)$ and $\Fc(3)$ has the unique F-equivalence class. 
Moreover, $\Fc(4)$ consists of three F-equivalence classes, 
one of which consists of 122 smooth Fano 4-polytopes being F-equivalent to $T^4$. 
Each of the others consists of one smooth Fano 4-polytope, which are $V^4$ and $\widetilde{V}^4$, respectively. 

Sato also conjectured that any smooth Fano $n$-polytope is either F-equivalent to $T^n$ or pseudo-symmetric. 
(\cite[Conjecture 1.3 and 6.3]{Sato00}). This is true when $n \leq 4$. 
However, {\O}bro \cite{Obro08} has given a counterexample of this conjecture. 
His counterexample is a smooth Fano 5-polytope $P$ with 8 vertices 
which is neither pseudo-symmetric nor I-equivalent to any other smooth Fano 5-polytope, 
i.e., for any $Q \in \Fc(5) \setminus \{P\}$, $Q$ is never I-equivalent to $P$. 

In this paper, for the further investigations of equivalence classes for $\Fc(n)$ 
with respect to both of F-equivalence and I-equivalence, we investigate smooth Fano 5-polytopes 
and determine all of F-equivalence classes as well as I-equivalence classes for $\Fc(5)$. 
Moreover, we introduce {\em F-isolated} or {\em I-isolated} smooth Fano polytope (See Section \ref{higher}) 
and we characterize completely F-isolated (I-isolated) smooth Fano $n$-polytopes with $n+3$ vertices. 
In addition, we construct a family of I-isolated smooth Fano $n$-polytopes for each $n \geq 5$. 


A brief organization of this paper is as follows. 
First, in Section \ref{junbi}, we recall some notion on smooth Fano polytopes, fix some notation 
and prepare some lemmas for the main results. Next, in Section \ref{5dim}, 
we describe all of F-equivalence classes as well as I-equivalence classes for $\Fc(5)$. 
Moreover, in Section \ref{higher}, we introduce F-isolated or I-isolated smooth Fano polytopes and 
we give a complete characterization of F-isolated (I-isolated) smooth Fano $n$-polytopes with $n+3$ vertices 
(Theorem \ref{pic3cha}). In addition, 
we construct a family of I-isolated smooth Fano $n$-polytopes 
with $n+\rho$ vertices for $n \geq 5$ and $3 \leq \rho \leq n$ (Theorem \ref{general} and Corollary \ref{kei}).


\section{Preliminaries}\label{junbi}


First, we note the following. 
For a smooth Fano polytope $P \subset \RR^n$, let $$\Sigma(P)=\{ \cn(F) \subset \RR^n : F \text{ is a face of }P\},$$ 
where $\cn(F)$ denotes the cone generated by $F$. 
Then $\Sigma(P)$ is a complete nonsingular fan. (For the terminologies on fans, conslut, e.g., \cite{Fulton}.)

Let $\Sigma$ be a complete nonsingular fan. 
We recall the useful notions, {\em primitive collections} and {\em primitive relations}, introduced by Batyrev \cite{Bat91}. 
Let $G(\Sigma)$ be the set of the primitive vectors of 1-dimensional cones in $\Sigma$. 
\begin{itemize}
\item We call a nonempty subset $A \subset G(\Sigma)$ a {\em primitive collection} of $\Sigma$ 
if $\cn(A)$ is not a cone in $\Sigma$ but $\cn(A \setminus \{x\})$ is a cone in $\Sigma$ for every $x \in A$. 
Let PC$(\Sigma)$ denote the set of all primitive collections of $\Sigma$. 
\item Let $A=\{x_1,\ldots,x_l\} \in \text{PC}(\Sigma)$. If $x_1+ \cdots + x_l \not= 0$, 
then there is a unique cone $\cn(\{y_1,\ldots,y_m\})$ in $\Sigma$ 
such that $x_1+\cdots+x_l= a_1y_1+\cdots+a_my_m,$ where $a_i \in \ZZ_{>0}$. 
We call the relation $x_1+ \cdots + x_l=0$ or $x_1+\cdots+x_l= a_1y_1+\cdots+a_my_m$ the {\em primitive relation} for $A$. 
\item For $A \in \text{PC}(\Sigma)$, if the primitive relation for $A$ is $x_1+\cdots+x_l= a_1y_1+\cdots+a_my_m$, 
then $\deg(A)=l - (a_1+\cdots+a_m)$ is called the {\em degree} of $A$. 
\end{itemize}
Notice that these definitions are also available to smooth Fano polytopes. 
In the case of smooth Fano polytopes $P$, we use a notation PC$(P)$ instead of PC$(\Sigma(P))$.


\begin{Proposition}[{\cite[Theorem 3.10]{Sato00}, \cite{Bat99}}]\label{tokuchou}
For a complete nonsingular fan $\Sigma$, 
the degrees of the primitive collections of $\Sigma$ are all positive 
if and only if there exists a smooth Fano polytope $P$ such that $\Sigma=\Sigma(P)$. 
\end{Proposition}

Moreover, the primitive relations in smooth Fano $n$-polytopes with $n+2$ or $n+3$ vertices 
are completely characterized as follows.

\begin{Proposition}[{\cite[Theorem 1]{Kl88}}]\label{pic2}
Let $P$ be a smooth Fano $n$-polytope with $n+2$ vertices and let $V(P)=\{v_1,\ldots,v_{n+2}\}$. 
Then the primitive relations in $P$ are (up to renumeration of the vertices) of the form 
\begin{align*}
&v_1+\cdots+v_k=0 \; \text{ for some } \; 2 \leq k \leq n, \\
&v_{k+1}+\cdots + v_{n+2} = a_1v_1+\cdots+a_kv_k \; \text{ with }a_i \geq 0 \text{ and } n+2-k > a_1+\cdots+a_k. 
\end{align*}
\end{Proposition}

\begin{Proposition}[{\cite[Theorem 6.6]{Bat91}}]\label{pic3}
Let $P$ be a smooth Fano $n$-polytope with $n+3$ vertices. Then one of the following holds: 
\begin{itemize}
\item[(i)] $P$ consists of three disjoint primitive collections, i.e., 
$\text{{\em PC}}(P)=\{A_1,A_2,A_3\}$ with $A_i \cap A_j = \emptyset$ for each $i \not= j$; 
\item[(ii)] $|\text{{\em PC}}(P)|=5$ and there is $(p_0,p_1,p_2,p_3,p_4) \in \ZZ_{>0}^5$ with $p_0+\cdots+p_4=n+3$ 
such that the primitive relations in $P$ are of the forms
\begin{equation}\label{pic3prim}
\begin{aligned}
v_1+\cdots+v_{p_0}+y_1+\cdots+y_{p_1}&=c_2z_2+\cdots+c_{p_2}z_{p_2}+(d_1+1)t_1+\cdots+(d_{p_3}+1)t_{p_3}, \\
y_1+\cdots+y_{p_1}+z_1+\cdots+z_{p_2}&=u_1+\cdots+u_{p_4}, \\
z_1+\cdots+z_{p_2}+t_1+\cdots+t_{p_3}&=0, \\
t_1+\cdots+t_{p_3}+u_1+\cdots+u_{p_4}&=y_1+\cdots+y_{p_1}, \\
u_1+\cdots+u_{p_4}+v_1+\cdots+v_{p_0}&=c_2z_2+\cdots+c_{p_2}z_{p_2}+d_1t_1+\cdots+d_{p_3}t_{p_3}, 
\end{aligned}
\end{equation}
where $V(P)=\{v_1,\ldots,v_{p_0},y_1,\ldots,y_{p_1},z_1,\ldots,z_{p_2},t_1,\ldots,t_{p_3},u_1,\ldots,u_{p_4}\}$ 
and $c_2,\ldots,c_{p_2},d_1,\ldots,d_{p_3} \in \ZZ_{\geq 0}$ and 
the degree of each of these primitive collections is positive. 
\end{itemize}
\end{Proposition}

We also recall the following useful lemma.

\begin{Lemma}[{\cite[Lemma 3.1]{Obro08}, see also \cite[Corollary 4.4]{Cas03}}]\label{ext}
Let 
\begin{equation}\label{ex1}v_1+\cdots+v_k=a_1w_1+\cdots+a_mw_m\end{equation} 
be a linear relation of some vertices of a smooth Fano polytope $P$ such that 
$a_i \in \ZZ_{>0}$ and $\{v_1,\ldots,v_k\} \cap \{w_1,\ldots,w_m\} = \emptyset$. 
Suppose that $k-(a_1+\cdots+a_m)=1$ and $\con(\{w_1,\ldots,w_m\})$ is a face of $P$. 
Then \eqref{ex1} is a primitive relation. Moreover, for each face $F$ of $P$ with $\{w_1,\ldots,w_m\} \subset V(F)$, 
$\con((V(F) \cup \{v_1,\ldots,v_k\}) \setminus \{v_i\})$ is also a face of $P$ for every $1 \leq i \leq k$. 
\end{Lemma}


\section{Equivalence classes for smooth Fano 5-polytopes}\label{5dim}

In this section, we describe all F-equivalence or I-equivalence classes for $\Fc(5)$. 

\begin{Proposition}
The number of F-equivalence classes for $\Fc(5)$ is $27$ and 
the number of I-equivalence classes for $\Fc(5)$ is $4$. 
\end{Proposition}

The following Figure \ref{figure} shows all of the smooth Fano 5-polytopes which are not F-equivalent to $T^5$. 
Each circled number corresponds to one smooth Fano 5-polytope with $5+\rho$ vertices for $3 \leq \rho \leq 7$ 
and the number is the ID of the database ``Graded Ring Database'' of smooth Fano polytopes, 
which is based on the algorithm by {\O}bro (\cite{Obro07, Obrophd}). See 
\url{http://grdb.lboro.ac.uk/forms/toricsmooth} 

There are 2 smooth Fano 5-polytopes with 8 vertices, 
12 ones with 9 vertices, 16 ones with 10 vertices, 6 ones with 11 vertices and 2 ones with 12 vertices 
which are not F-equivalent to $T^5$. See also the table \ref{ro-}. 
Moreover, each of the double circled numbers corresponds to a smooth Fano 5-polytope $P$ such that 
for any $Q \in \Fc(5) \setminus \{P\}$, $Q$ is not I-equivalent to $P$ (i.e., I-isolated, see Section \ref{higher}). 
Note that {\O}bro's example \cite{Obro08} corresponds to the double circled number 164. 
In addition, if two circled numbers are connected by a line, then those are F-equivalent. 
We can see that there are 26 ``connected components'' in Figure \ref{figure}, 
each of which corresponds to an F-equivalence class for $\Fc(5)$. On the other hand, 
all of the smooth Fano 5-polytopes corresponding to the single circled numbers are I-equivalent to $T^5$.

Although there are only 2 smooth Fano 4-polytopes which are not F-equivalent to $T^4$, 
which are $V^4$ and $\widetilde{V^4}$ (see \cite{Sato00}), there are 38 smooth Fano 5-polytopes 
which are not F-equivalent to $T^5$, i.e., there are 38 circled numbers in Figure \ref{figure}. 

\bigskip

\begin{figure}[htb!]
\centering
\includegraphics[scale=0.3]{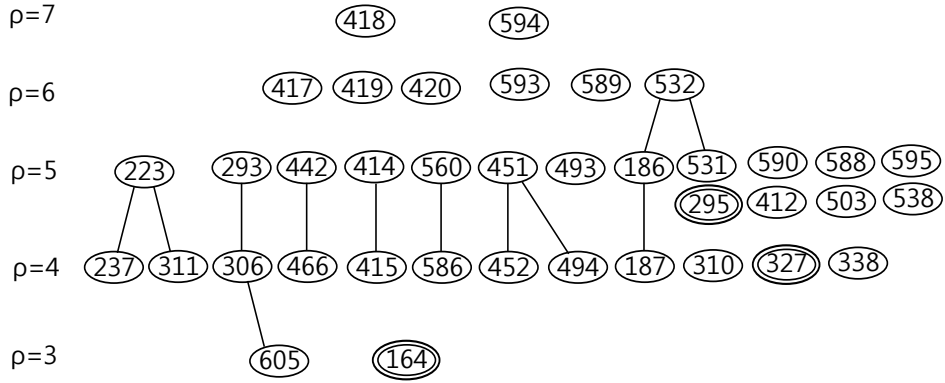}
\caption{smooth Fano 5-polytopes which are not F-equivalent to $T^5$}\label{figure}
\end{figure}

\begin{table}[htb]
\begin{tabular}{|c|c|c|c|c|c|}
\hline
\text{$\rho$}                      &3 &4  &5  &6 &7 \\ \hline
\text{smooth Fano $5$-polytopes}   &2 &12 &16 &6 &2 \\ \hline
\end{tabular}

\smallskip

\caption{The number of smooth Fano 5-polytopes with $5+\rho$ vertices 
which are not F-equivalent to $T^5$}\label{ro-}
\end{table}


\section{I-isolated smooth Fano polytopes}\label{higher}

We introduce the notions, F-isolated and I-isolated, for smooth Fano polytopes. 
\begin{Definition}
Let $P$ be a smooth Fano $n$-polytope. 
We say that $P$ is {\em F-isolated} (resp. {\em I-isolated}) 
if for any $P' \in \Fc(n) \setminus \{P\}$, $P'$ is not F-equivalent (resp. I-equivalent) to $P$. 
\end{Definition}
Obviously, $P$ is F-isolated if $P$ is I-isolated.


First, we characterize the primitive relations in I-isolated smooth Fano $n$-polytopes 
with $n+3$ vertices. Note that such a polytope is of dimension at least 5.

\begin{Theorem}\label{pic3cha}
Let $P$ be a smooth Fano $n$-polytope with $n+3$ vertices. 
Then the following three conditions are equivalent: 
\begin{itemize}
\item[(a)] $P$ is F-isolated; 
\item[(b)] $P$ is I-isolated; 
\item[(c)] 
All of the primitive relations in $P$ are of the forms
\begin{equation}\label{primrel}
\begin{aligned}
v_1+\cdots+v_a+y_1+\cdots+y_b&=(a+b-1)t, \\
y_1+\cdots+y_b+z&=u_1+\cdots+u_b, \\
z+t&=0, \\
t+u_1+\cdots+u_b&=y_1+\cdots+y_b, \\
u_1+\cdots+u_b+v_1+\cdots+v_a&=(a+b-2)t, 
\end{aligned}
\end{equation}
where $V(P)=\{v_1,\ldots,v_a,y_1,\ldots,y_b,z,t,u_1,\ldots,u_b\}$, 
$a \geq 2$ and $b \geq 2$ with $n=a+2b-1$. 
\end{itemize}
\end{Theorem}
\begin{proof}
{\bf ((a) $\Longrightarrow$ (c))} 
Assume that $P$ is F-isolated. Since $|V(P)|=n+3$, from \cite[Corllary 6.13]{Sato00}, 
we may assume $P$ satisfies the conditions in Proposition \ref{pic3} (ii), i.e., 
the primitive relations in $P$ are of the forms \eqref{pic3prim}. 

In the discussions below, by using the description of the primitive relations in 
a complete nonsingular fan obtained by introducing a new lattice point $w=\sum_{v \in G(\sigma)} v$ 
for some $\sigma \in \Sigma(P)$ and taking a stellar subdivision with the 1-dimensional cone generated by $w$, 
we prove that if there is no smooth Fano polytope $P'$ which is F-equivalent to $P$, 
then the primitive relations in $P$ is of the forms \eqref{primrel}. 
The description is explicitly given in \cite[Theorem 4.3]{Sato00}. 

For a face $F \subset P$, let $\Sigma_F^*(P)$ denote the new complete nonsingular fan 
obtained by taking a stellar subdivision of $\Sigma(P)$ with $\cn(\{w\})$, 
where $w = \sum_{v \in V(F)}v ( = \sum_{v \in G(\cn(F))}v)$. 
\begin{itemize}
\item Suppose $p_1=1$ or $p_4=1$. Then by \cite[Proposition 8.3]{Sato00}, we obtain a new smooth Fano $n$-polytope 
with $n+2$ vertices which is F-equivalent to $P$, a contradiction. Hence $p_1 \geq 2$ and $p_4 \geq 2$. 
\item Suppose $p_0=1$. Since $\{v_1,z_1\}$ is not contained in PC$(P)$, $\con(\{v_1,z_1\})$ is a face of $P$. 
We consider a stellar subdivision of $\Sigma(P)$ with a lattice point $w=v_1+z_1$ 
and the complete nonsingular fan $\Sigma_{\con(\{v_1,z_1\})}^*(P)$. 
Then the new primitive relations in $\Sigma_{\con(\{v_1,z_1\})}^*(P)$ concerning $w$ are 
\begin{align*}
&w+y_1+\cdots+y_{p_1}=(c_2-1)z_2+\cdots+(c_{p_2}-1)z_{p_2}+d_1t_1+\cdots+d_{p_3}t_{p_3}, \\
&w+z_2+\cdots+z_{p_2}+t_1+\cdots+t_{p_3}=v_1, \\
&u_1+\cdots+u_{p_4}+w=(c_2-1)z_2+\cdots+(c_{p_2}-1)z_{p_2}+(d_1-1)t_1+\cdots+(d_{p_3}-1)t_{p_3}. 
\end{align*}
(See \cite[Theorem 4.3]{Sato00}.) Since the degree of each of these new primitive collections is positive, 
by Proposition \ref{tokuchou}, there is a smooth Fano polytope $P'$ with $\Sigma_{\con(\{v_1,z_1\})}^*(P)=\Sigma(P')$, 
and in particular, $P'$ is F-equivalent to $P$, a contradiction. Hence $p_0 \geq 2$. 
\item Suppose $p_1+p_2 - p_4 \geq 2$. Then we see that the complete nonsingular fan $\Sigma_{\con(\{y_1,z_1\})}^*(P)$ 
comes from a smooth Fano polytope, a contradiction. Thus $p_1+p_2 - p_4 \leq 1$. 
On the other hand, since $\{y_1,\ldots,y_{p_1},z_1,\ldots,z_{p_2}\}$ is a primitive collection of $P$ and 
its degree is equal to $p_1 + p_2 - p_4$, we obtain that $p_1 + p_2 - p_4 \geq 1$ by Propoisition \ref{tokuchou}. 
Hence $p_1+p_2-p_4=1$. 
\item Similarly, suppose $p_3+p_4-p_1 \geq 2$. Then we see that the complete nonsingular fan $\Sigma_{\con(\{t_1,u_1\})}^*$(P) 
comes from a smooth Fano polytope, a contradiction. Thus $p_3 +  p_4- p_1 \leq 1$. 
Moreover, since $\{t_1,\ldots,t_{p_3},u_1,\ldots,u_{p_4}\}$ is a primitive collection of $P$ and 
its degree is equal to $p_3+p_4-p_1$, we have $p_3+p_4-p_1 \geq 1$ by Proposition \ref{tokuchou}. Hence $p_3+p_4-p_1 = 1$. 
\item Hence, in particular, we obtain $p_1+p_2-p_4=p_3+p_4-p_1=1$. This implies that $p_2=p_3=1$ and $p_1=p_4$. 
\item Suppose $p_0+p_1-(d_1+1) \geq 2$. Then we see that the complete nonsingular fan $\Sigma_{\con(\{v_1,v_2\})}^*(P)$ 
comes from a smooth Fano polytope, a contradiction. Thus $p_0+p_1-(d_1+1) \leq 1$. 
On the other hand, since $\{v_1,\ldots,v_{p_0},y_1,\ldots,y_{p_1}\}$ is a primitive collection of $P$ and 
its degree is equal to $p_0+p_1 - (d_1+1)$, we obtain that $p_0+p_1 - (d_1+1) \geq 1$ by Proposition \ref{tokuchou}. 
Hence $d_1=p_0+p_1-2$. 
\end{itemize}
By summarizing these, we obtain the desired primitive relations \eqref{primrel}. 

\smallskip

{\bf (c) $\Longrightarrow$ (b)} Let 
\begin{align*}
&v_i=\eb_i, \; i=1,\ldots,a-1, \;\; 
v_a=-\eb_1-\cdots-\eb_{a+2b-2}+(a+b-1) \eb_{a+2b-1}, \\
&y_j=\eb_{a-1+j}, \; j=1,\ldots,b-1, \;\; y_b=\eb_{a+b-1}+\cdots+\eb_{a+2b-2}, \\
&z=-\eb_{a+2b-1}, \;\; t=\eb_{a+2b-1}, \\
&u_j=\eb_{a+b-1+j}, \; j=1,\ldots,b-1, \;\; u_b=\eb_a+\cdots+\eb_{a+b-1}-\eb_{a+2b-1}, 
\end{align*}
let $V=\{v_1,\ldots,v_a\}, Y=\{y_1,\ldots,y_b\}, U=\{u_1,\ldots,u_b\}$ 
and let $P=\con( V \cup Y \cup \{z,t\} \cup U)$. Then the primitive relations in $P$ are 
of the forms \eqref{primrel}. Our work is to prove that this $P$ is I-isolated. 

\noindent\underline{$\not\exists Q \subset P$} : 
First, we prove there is no smooth Fano $n$-polytope $Q$ with $n+2$ vertices such that $V(Q) \subset V(P)$. 

Suppose, on the contrary, that there is a smooth Fano $n$-polytope $Q$ with $n+2$ vertices 
such that $V(Q)=V(P) \setminus \{w\}$, where $w \in V(P)=V \cup Y \cup \{z,t\} \cup U$. 

Assume that $w \in \{z,t\}$. In this case, although there must be a primitive collection $A \in$ PC$(Q)$ 
with $\sum_{v \in A}v=0$ by Proposition \ref{pic2}, no non-empty subset of $V(Q)=V(P) \setminus \{w\}$ add to 0, a contradiction. 

Assume that $w \in V$. Then $w=v_i$ for some $1 \leq i \leq a$. 
On the other hand, for each $1 \leq j \leq a-1$ (resp. $j=a$), 
the $j$th entry of each vertex of $P$ except for $v_j$ is nonpositive (resp. nonnegative). 
Thus $Q$ cannot contain the origin in its interior, a contradiction. 

Assume that $w \in Y \cup U$. Let $A$ be a primitive collection of $Q$ with $\sum_{v \in A}v=0$. 
Then $A=\{z,t\}$. By Proposition \ref{pic2}, $\sum_{x \in (V \cup Y \cup U) \setminus \{w\}}x$ 
should be written as a linear combination of $z$ and $t$, a contradiction. 


\noindent\underline{$\not\exists R \supset P$} : 
Next, we prove there is no smooth Fano $n$-polytope $R$ with $n+4$ vertices such that $V(P) \subset V(R)$. 

Suppose that there is a smooth Fano $n$-polytope $R$ with $n+4$ vertices 
such that $V(R)=V(P) \cup \{w\}$ for some new lattice point $w$. 

Since $t$ is a vertex of $R$, by Lemma \ref{ext}, the relation 
\begin{equation}\label{1}
v_1+\cdots+v_a+y_1+\cdots+y_b=(a+b-1)t
\end{equation}
is a primitive relation in $R$. Hence $\con((V \setminus \{v\}) \cup Y \cup \{t\} )$ is a face of $R$ 
for every $v \in V$. Since we have the relation 
\begin{equation}\label{2} t+u_1+\cdots+u_b=y_1+\cdots+y_b, \end{equation}
by Lemma \ref{ext}, we obtain that \eqref{2} is also a primitive relation in $R$ and 
$$\con((V \setminus \{v\}) \cup Y \cup \{t\} \cup (U \setminus \{u\})) \text{ and } 
\con((V \setminus \{v\}) \cup Y \cup U)$$ 
are the facets of $R$ for every $v \in V$ and $u \in U$. Moreover, 
since $\con((V \setminus \{v\}) \cup Y \cup \{t\} \cup (U \setminus \{u\}))$ is a face of $R$, 
we obtain that $$\con(V \cup (Y \setminus\{y\}) \cup \{t\} \cup (U \setminus \{u\}))$$ 
is a facet of $R$ for every $y \in Y$ and $u \in U$ by \eqref{1}. In addition, from the relation 
\begin{equation}\label{3} y_1+\cdots+y_b+z=u_1+\cdots+u_b, \end{equation}
since $\con((V \setminus \{v\}) \cup Y \cup U)$ is a face of $R$, we obtain that 
\eqref{3} is also a primitive relation in $R$ and 
$$\con((V \setminus \{v\}) \cup (Y \setminus \{y\}) \cup \{z\} \cup U)$$ 
is also a facet of $R$ for every $v \in V$ and $y \in Y$. 

Therefore, $R$ contains the following four kinds of facets: 
\begin{align*}
&\con((V \setminus \{v\}) \cup Y \cup \{t\} \cup (U \setminus \{u\})), \text{ where } v \in V \text{ and } u \in U; \\
&\con((V \setminus \{v\}) \cup Y \cup U), \text{ where } v \in V; \\
&\con(V \cup (Y \setminus\{y\}) \cup \{t\} \cup (U \setminus \{u\})), \text{ where } y \in Y \text{ and } u \in U; \\
&\con((V \setminus \{v\}) \cup (Y \setminus \{y\}) \cup \{z\} \cup U), \text{ where } v \in V \text{ and } y \in Y. 
\end{align*}
These are also the facets of $P$. Thus, for any these facets $F$, $w$ is not contained in $\cn(F)$. 
Therefore, $w$ is contained in the cone generated by the remaining facet of $P$, i.e., 
$$w \in \cn(V \cup (Y \setminus \{y\}) \cup \{z\} \cup (U \setminus \{u\}))$$ 
for some $y \in Y$ and $u \in U$. Without loss of generality, we may assume that 
$w \in \cn(V \cup (Y \setminus \{y_b\}) \cup \{z\} \cup (U \setminus \{u_b\}))$. Let 
\begin{align*}
w=c_1v_1+\cdots+c_{a-1}v_{a-1}+c_ay_1+\cdots+&c_{a+b-2}y_{b-1}+c_{a+b-1}v_a+ \\
&c_{a+b}u_1+\cdots+c_{a+2b-2}u_{b-1}+c_{a+2b-1}z, 
\end{align*}
where $c_i \in \ZZ_{\geq 0}$. Let $F=\con(V \cup Y' \cup \{t\} \cup U')$, 
where $Y'=Y \setminus \{y_b\}$ and $U'=U \setminus \{u_b\}$. Then $F$ is a facet of $R$. 
Let $F'$ be the unique facet of $R$ such that $F \cap F'$ is a ridge (i.e. the face of dimension $n-2$) 
of $R$ with $V(F) \setminus V(F')=\{t\}$. Then it must be satisfied that $V(F') \setminus V(F) = \{w\}$. 
In fact, for each $w' \in \{y_b,u_b\}$, $\con(V \cup Y' \cup U' \cup \{w'\})$ 
cannot be a face of $R$ because $\con(V \cup Y)$ and $\con(V \cup U)$ cannot be a face, 
and for $z$, $\con(V \cup Y' \cup \{z\} \cup U')$ is not a face by our assumption. 
Hence, by \cite[Lemma 2.1]{Obro08}, $t+w$ is in the linear subspace spanned by $V \cup Y' \cup U'$. 
Therefore, from the relation $z+t=0$, we have $c_{a+2b-1}=1$. Moreover it follows from \cite[Lemma 2.1]{Obro08} again that 
we have $$1 > \langle \ab_F, w \rangle = \sum_{i=1}^{a+2b-2} c_i - 1 > \langle \ab_F, z \rangle = -1,$$ 
where $\ab_F$ is the lattice vector defining $F$, i.e., $F=\{\alpha \in P : \langle \ab_F,\alpha \rangle = 1\}$. 
Thus $\sum_{i=1}^{a+2b-2} c_i=1$. 
Namely, $w$ can be written like $w=z + w''$ with some $w'' \in V \cup Y' \cup U'$. 
Let $G=\con((V \setminus \{v\}) \cup Y \cup U)$. Since $G$ is a facet of $R$ for each $v \in V$ 
and the relation $y_1+\cdots+y_b+w = u_1+\cdots+u_b + w''$ holds, 
we see that $w$ is also contained in $G$, i.e., $R$ is not simplicial, a contradiction. 


{\bf((b) $\Longrightarrow$ (a))} This is obvious. 
\end{proof}

Next, we provide a family of I-isolated smooth Fano polytopes.

\begin{Theorem}\label{general}
Let $a \geq 2$, $b \geq 1$, $k \geq 1$, and $l_j \geq 1$ for $j=1,\ldots,k$ be integers. 
Then there exists an I-isolated (in particular, F-isolated) smooth Fano polytope $P$ 
of dimension $n=a+2b-1+\sum_{j=1}^k l_j$ with $n+k+3$ vertices whose primitive relations are of the forms
\begin{equation}
\begin{aligned}\label{primrel1}
v_1+\cdots+v_a+y_1+\cdots+y_b&=(a+b-1)t, \\
y_1+\cdots+y_b+z&=u_1+\cdots+u_b, \\
z+t&=0, \\
t+u_1+\cdots+u_b&=y_1+\cdots+y_b, \\
u_1+\cdots+u_b+v_1+\cdots+v_a&=(a+b-2)t, \\
w_{1,1}+\cdots+w_{l_1+1,1}&=\alpha_{1,1}+\cdots+\alpha_{l_1,1}, \\
w_{1,2}+\cdots+w_{l_2+1,2}&=\alpha_{1,2}+\cdots+\alpha_{l_2,2}, \\
&\cdots \\
w_{1,k}+\cdots+w_{l_k+1,k}&=\alpha_{1,k}+\cdots+\alpha_{l_k,k}, 
\end{aligned}
\end{equation}
where $V(P)=\{v_1,\ldots,v_a,y_1,\ldots,y_b,z,t,u_1,\ldots,u_b\} \cup \bigcup_{j=1}^k \{w_{1,j},\ldots,w_{l_j+1,j}\}$ 
and $\{\alpha_{i,j} : 1 \leq i \leq l_j, 1 \leq j \leq k\} = \{y_q,u_q : 1 \leq q \leq b\}$. 
\end{Theorem}
\begin{proof}
Let 
\begin{align*}
&v_i=\eb_i, \; i=1,\ldots,a-1, \;\; 
v_a=-\eb_1-\cdots-\eb_{a+2b-2}+(a+b-1) \eb_{a+2b-1}, \\
&y_j=\eb_{a-1+j}, \; j=1,\ldots,b-1, \;\; y_b=\eb_{a+b-1}+\cdots+\eb_{a+2b-2}, \\
&z=-\eb_{a+2b-1}, \;\; t=\eb_{a+2b-1}, \\
&u_j=\eb_{a+b-1+j}, \; j=1,\ldots,b-1, \;\; u_b=\eb_a+\cdots+\eb_{a+b-1}-\eb_{a+2b-1}, \\
&w_{i,j}=\eb_{a+2b-1+\sum_{q=1}^{j-1}l_q+i}, \; i=1,\ldots,l_j, \; j=1,\ldots,k, \\
&w_{l_j+1,j}=-\eb_{a+2b-1+\sum_{q=1}^{j-1}l_q+1}-\cdots-\eb_{a+2b-1+\sum_{q=1}^j l_q}+\sum_{i=1}^{l_j} \alpha_{i,j}, \; j=1,\ldots,k, 
\end{align*}
let $V=\{v_1,\ldots,v_a\}, Y=\{y_1,\ldots,y_b\}, U=\{u_1,\ldots,u_b\}$ and 
let $W_j=\{w_{i,j}: 1 \leq i \leq l_j+1\}$ for each $1 \leq j \leq k$ and $W= \bigcup_{j=1}^k W_j$. 
We define $P = \con( V \cup Y \cup \{z,t\} \cup U \cup W)$. 
Then $P$ is a smooth Fano polytope of dimension $n=a+2b-1+\sum_{i=1}^k l_j$ with $n+k+3$ vertices and 
all of its primitive relations are of the forms \eqref{primrel1}. 
By Lemma \ref{hodai1} and Lemma \ref{hodai2} below, we see that $P$ is I-isolated. 
\end{proof}

\begin{Lemma}\label{hodai1}
Let $P$ be the polytope given in the proof of Theorem \ref{general}. 
Then there exists no smooth Fano $n$-polytope $Q$ with $n+k+2$ vertices such that $V(Q) \subset V(P)$. 
\end{Lemma}
\begin{proof}
Work with the same notation as in the proof of Theorem \ref{general}. 
Suppose that there is a smooth Fano $n$-polytope $Q$ with $n+k+2$ vertices 
such that $V(Q)=V(P) \setminus \{x\}$, where $x \in V(P)=V \cup Y \cup \{z,t\} \cup U \cup W$. 

Assume that $x \in \{z,t\}$. In this case, although there must be a primitive collection $A \in$ PC$(Q)$ with $\sum_{v \in A}v=0$ 
by Proposition \ref{pic2}, no non-empty subset of $V(P) \setminus \{x\}$ add to 0, 
a contradiction. 

Assume that $x \in V$. Then $x=v_i$ for some $1 \leq i \leq a$. On the other hand, for each $1 \leq j \leq a-1$ (resp. $j=a$), 
the $j$th entry of each vertex of $P$ except for $v_j$ is nonpositive (resp. nonnegative). 
Thus $Q$ cannot contain the origin in its interior, a contradiction. 
Similarly, if $x \in W$, then $Q$ cannot contain the origin in its interior, a contradiction. 

Assume that $x \in Y$. Without loss of generality, we assume $x=y_b$. 
For each $1 \leq j \leq k$, let $m_j=|\{ i : \alpha_{i,j}=y_b, 1 \leq i \leq l_j\}|$ and 
let $m_j'=|\{ i : \alpha_{i,j}=u_b, 1 \leq i \leq l_j\}|$. 
Let $r$ be an index attaining $\max\left\{ \frac{m_j-1}{m_j'+1} : 1 \leq j \leq k \right\}$. 
Consider $F=\con((V \setminus \{v_a\}) \cup (Y \setminus \{y_b\}) \cup \{z\} \cup (U \setminus \{u_b\}) \cup 
(W \setminus \{w_{l_j+1,j} : 1 \leq j \leq k, j \not=r\}))$. Then $F$ is a face of $Q$. In fact, let 
$${\bf a}=\sum_{1 \leq q \leq n, q \not=a+b-1} \eb_q +\left(\frac{1+m_r'}{m_r+m_r'}-b+2\right)\eb_{a+b-1}.$$ 
Then we see the following: 
\begin{itemize}
\item We have $\langle \ab, v_i \rangle = 1$ for each $1 \leq i \leq a-1$, $\langle \ab, t \rangle=1$, 
$\langle \ab, y_j \rangle = \langle \ab, u_j \rangle = 1$ for each $1 \leq j \leq b-1$ and 
$\langle \ab, w_{i,j} \rangle=1$ for each $1 \leq i \leq l_j$ and $1 \leq j \leq k$. 
\item Since $w_{l_r+1,r}=\sum_{j=1}^{l_r} (\alpha_{j,r}-w_{j,r})$, we have 
\begin{align*}
\langle \ab, w_{l_r+1,r} \rangle &= \sum_{j=1}^{l_r} \langle \ab, \alpha_{j,r} \rangle - \sum_{j=1}^{l_r} \langle \ab, w_{j,r} \rangle \\
&=(m_r \langle \ab, y_b \rangle + m_r' \langle \ab, u_b \rangle + (l_r-m_r-m_r')) - l_r \\
&=m_r \left(1+\frac{1+m_r'}{m_r+m_r'}\right) + m_r' \frac{1+m_r'}{m_r+m_r'} - m_r -m_r'=1. 
\end{align*}
\item We have $\langle \ab, v_a \rangle= -\frac{1+m_r'}{m_r+m_r'} < 1$ and $\langle \ab, z \rangle = -1 < 1$. 
\item Moreover, we have $\langle \ab, u_b \rangle = \frac{1+m_r'}{m_r+m_r'}$. 
Thus $\langle \ab, u_b \rangle=1$ if $m_r=1$, otherwise $\langle \ab, u_b \rangle < 1$. 
\item In addition, for each $1 \leq j \leq k$, we also have 
\begin{align*}
\langle \ab, w_{l_j+1,j} \rangle &= \sum_{i=1}^{l_j} \langle \ab,\alpha_{i,j} \rangle - \sum_{i=1}^{l_j} \langle \ab, w_{i,j}\rangle  \\
&=m_j\left( 1 + \frac{1+m_r'}{m_r+m_r'}\right) + m_j'\frac{1+m_r'}{m_r+m_r'} +(l_j- m_j - m_j')-l_j \\
&=\frac{m_j(m_r'+1)-m_j'(m_r-1)}{m_r+m_r'} \leq \frac{m_r+m_r'}{m_r+m_r'} =1. 
\end{align*}
Here, the inequality $m_j(m_r'+1)-m_j'(m_r-1) \leq m_r+m_r'$ follows from the the maximality of $\frac{m_r-1}{m_r'+1}$, 
i.e., this inequality is equivalent to $\frac{m_r-1}{m_r'+1} \geq \frac{m_j-1}{m_j'+1}$. 
\end{itemize}
If $m_r=1$, then $F$ contains at least $n+1$ vertices since $u_b \in F$. Hence, $Q$ is not simplicial, a contradiction. 
If $m_r > 1$, then $\frac{1+m_r'}{m_r+m_r'}$ is not an integer. Hence, $Q$ is not reflexive, a contradiction.

Assume that $x \in U$. Without loss of generality, we assume $x=u_b$. 
For each $1 \leq j \leq k$, let $m_j$ and $m_j'$ as above. 
Let $s$ be the index attaining $\max\left\{ \frac{m_j'-1}{m_j+1} : 1 \leq j \leq k \right\}$. 
Consider $G=\con((V \setminus \{v_a\}) \cup (Y \setminus \{y_b\}) \cup \{z\} \cup (U \setminus \{u_b\}) \cup 
(W \setminus \{w_{l_j+1,j} : 1 \leq j \leq k, j\not=s\}))$. Then $G$ is a face of $Q$. In fact, let 
$${\bf b}=\sum_{1 \leq q \leq n, q \not=a+b-1, q\not=a+2b-1} \eb_q +\left(\frac{1+m_s}{m_s+m_s'}-b+1 \right)\eb_{a+b-1}-\eb_{a+2b-1}.$$ 
Then we see the following: 
\begin{itemize}
\item We have $\langle \bb, v_i \rangle = 1$ for each $1 \leq i \leq a-1$, $\langle \bb, z \rangle=1$, 
$\langle \bb, y_j \rangle = \langle \bb, u_j \rangle = 1$ for each $1 \leq j \leq b-1$ and 
$\langle \bb, w_{i,j} \rangle=1$ for each $1 \leq i \leq l_j$ and $1 \leq j \leq k$. 
\item 
We have 
\begin{align*}
\langle \bb, w_{l_s+1,s} \rangle &= \sum_{j=1}^{l_s} \langle \bb, \alpha_{j,s} \rangle - \sum_{j=1}^{l_s} \langle \bb, w_{j,s} \rangle \\
&=(m_s \langle \bb, y_b \rangle + m_s' \langle \bb, u_b \rangle + (l_s-m_s-m_s')) - l_s \\
&=m_s \frac{1+m_s}{m_s+m_s'} + m_s' \left( 1+ \frac{1+m_s}{m_s+m_s'}\right) - m_s -m_s'=1. 
\end{align*}
\item We have $\langle \bb, v_a \rangle < -2a-2b+3 < 1$ and $\langle \bb, t \rangle = -1 < 1$. 
\item Moreover, we have $\langle \bb, y_b \rangle = \frac{1+m_s}{m_s+m_s'}$. 
Thus we see that $\langle \bb, y_b \rangle=1$ if $m_s'=1$, otherwise $\langle \bb, y_b \rangle < 1$. 
\item In addition, for each $1 \leq j \leq k$, we also have 
\begin{align*}
\langle \bb, w_{l_j+1,j} \rangle &= \sum_{i=1}^{l_j} \langle \bb,\alpha_{i,j} \rangle - \sum_{i=1}^{l_j} \langle \bb, w_{i,j}\rangle  \\
&=m_j \frac{1+m_s}{m_s+m_s'} + m_j' \left( 1 + \frac{1+m_s}{m_s+m_s'} \right) +(l_j- m_j - m_j')-l_j \\
&=\frac{m_j'(m_s+1)-m_j(m_s'-1)}{m_s+m_s'} \leq \frac{m_s+m_s'}{m_s+m_s'} =1. 
\end{align*}
Here, the inequality $m_j'(m_s+1)-m_j(m_s'-1) \leq m_s+m_s'$ follows from the the maximality of $\frac{m_s'-1}{m_s+1}$, 
i.e., this inequality is equivalent to $\frac{m_s'-1}{m_s+1} \geq \frac{m_j'-1}{m_j+1}$. 
\end{itemize}
If $m_s'=1$, then $G$ contains at least $n+1$ vertices since $y_b \in G$. Hence, $Q$ is not simplicial, a contradiction. 
If $m_s' > 1$, then $\frac{1+m_s}{m_s+m_s'}$ is not an integer. Hence, $Q$ is not reflexive, a contradiction. 
%
\end{proof}

\begin{Lemma}\label{hodai2}
Let $P$ be the polytope given in the proof of Theorem \ref{general}. 
Then there exists no smooth Fano $n$-polytope $R$ with $n+k+4$ vertices such that $V(P) \subset V(R)$. 
\end{Lemma}
\begin{proof}
Suppose that there is a smooth Fano $n$-polytope $R$ with $n+k+4$ vertices such that 
$V(R)=V(P) \cup \{x\}$, where $x$ is some new lattice point. 

Since the linear relation 
\begin{align}\label{r1}
v_1+\cdots+v_a+y_1+\cdots+y_b=(a+b-1)t
\end{align}
among the vertices of $R$ holds and $t$ is a vertex of $R$, by Lemma \ref{ext}, this is a primitive relation in $R$. 
Thus, $\con(Y)$ is a face of $R$. Hence, the relation 
\begin{align}\label{r2}
t+u_1+\cdots+u_b=y_1+\cdots+y_b
\end{align}
is also a primitive relation in $R$ by Lemma \ref{ext} again. Thus, $\con(U)$ is also a face of $R$. Hence, the relation 
\begin{align}\label{r3}
y_1+\cdots+y_b+z=u_1+\cdots+u_b
\end{align}
is also a primitive relation in $R$. Moreover, $\con(Y \cup U)$ is also a face of $R$ 
and so is $\con(T)$ for each $T \subset Y \cup U$. Hence, 
\begin{align}\label{rj}
w_{1,j} + \cdots + w_{l_j+1,j}=x_{1,j} + \cdots + x_{l_j,j}
\end{align}
is also a primitive relation in $R$ for each $1 \leq j \leq k$. 
Thus, by Lemma \ref{ext} again, 
$\con(Y \cup U \cup \bigcup_{j=1}^k (W_j \setminus \{w^{(j)}\}))$ is also a face of $R$ for each $w^{(j)} \in W_j$, 
and so is $\con(Y \cup \{t\} \cup (U \setminus \{u\}) \cup \bigcup_{j=1}^k (W_j \setminus \{w^{(j)}\}))$ for each $u \in U$ 
by \eqref{r2}. From \eqref{r1} and Lemma \ref{ext}, we obtain that 
$$\con((V \setminus \{v\}) \cup Y \cup \{t\} \cup (U \setminus \{u\}) \cup \bigcup_{j=1}^k (W_j \setminus \{w^{(j)}\}))$$
is a facet of $R$ for each $v \in V$, $u \in U$ and $w^{(j)} \in W_j$. 
By \eqref{r2} and Lemma \ref{ext}, we obtain that 
$$\con((V \setminus \{v\}) \cup Y \cup U \cup \bigcup_{j=1}^k (W_j \setminus \{w^{(j)}\}))$$ 
is also a facet of $R$. By the similar discussions, we obtain 
\begin{align*}
&\con(V \cup (Y \setminus \{y\}) \cup \{t\} \cup (U \setminus \{u\}) \cup \bigcup_{j=1}^k (W_j \setminus \{w^{(j)}\})) \text{ and }\\
&\con((V \setminus \{v\}) \cup (Y \setminus \{y\}) \cup \{z\} \cup U \cup \bigcup_{j=1}^k (W_j \setminus \{w^{(j)}\}))
\end{align*}
are also the facets of $R$ from \eqref{r1}, \eqref{r2}, \eqref{r3} and Lemma \ref{ext}. 

On the other hand, these four kinds of facets of $R$ are also the facets of $P$. 
In addition, $P$ contains one more kind of facet, which is of the form 
$$\con(V \cup (Y \setminus \{y\}) \cup \{z\} \cup (U \setminus \{u\}) \cup \bigcup_{j=1}^k (W_j \setminus \{w^{(j)}\})),$$
where $y \in Y$, $u \in U$ and $w^{(j)} \in W_j$. 

Therefore, $x$ must belong to the cone generated by this facet, i.e., 
\begin{align}\label{assume}
x \in \cn(V \cup (Y \setminus \{y\}) \cup \{z\} \cup (U \setminus \{u\}) \cup \bigcup_{j=1}^k (W_j \setminus \{w^{(j)}\})) 
\end{align}
for some $y \in Y$, $u \in U$ and $w^{(j)} \in W_j$. 

Without loss of generality, we may assume that 
$y=y_b, u=u_b$ and $w^{(j)}=w_{l_j+1,j}$ for each $1 \leq j \leq k$. Let 
\begin{align*}
x &= c_1v_1 + \cdots + c_{a-1}v_{a-1} + c_a y_1 + \cdots + c_{a+b-2}y_{b-1} + c_{a+b-1} v_a  \\
&+ c_{a+b} u_1 + \cdots + c_{a+2b-2} u_{b-1} + c_{a+2b-1}z + c_{a+2b}w_{1,1} + c_{a+2b+1} w_{2,1} + \cdots + c_n w_{l_k,k}. 
\end{align*}
Let $F=\con(V \cup Y' \cup \{t\} \cup U' \cup \bigcup_{j=1}^k W_j')$, 
where $Y'=Y \setminus \{y_b\}$, $U'=U \setminus \{u_b\}$ and $W_j'=W_j \setminus \{w_{l_j+1,j}\}$ for $1 \leq j \leq k$. 
Then $F$ is a facet of $R$. Let $F'$ be the unique facet of $R$ such that $F \cap F'$ 
is a ridge of $R$ with $V(F) \setminus V(F')=\{t\}$. 
Then it must be satisfied that $V(F') \setminus V(F) = \{x\}$. In fact, 
by considering the primitive relations \eqref{primrel1} and the assumption \eqref{assume}, 
the vertices $y_b, z, u_b, w_{l_j+1,j}$ cannot belong to $V(F')$. 
Thus, by \cite[Lemma 2.1]{Obro08}, $x+t$ is in the linear subspace spanned by 
$V \cup Y' \cup U' \cup \bigcup_{j=1}^k W_j'$. 
Therefore, from the relation $z+t=0$, we have $c_{a+2b-1}=1$. Moreover, it follows from \cite[Lemma 2.1]{Obro08} again that 
$$1 > \langle \ab_F, x \rangle = \sum_{1 \leq i \leq n, i \not=a+2b-1} c_i - 1 > \langle \ab_F, z \rangle = -1,$$ 
where $\ab_F$ is the lattice vector defining $F$, i.e., $F=\{\alpha \in P : \langle \ab_F, \alpha \rangle = 1\}$. 
Thus, $\sum_{1 \leq i \leq n, i \not=a+2b-1} c_i=1$. Namely, $x$ can be written like $x=z + \gamma$ 
with some $\gamma \in V \cup Y' \cup U' \cup \bigcup_{j=1}^k W_j'$. 

Now we have the relation $y_1+\cdots+y_b+x=u_1+\cdots+u_b+\gamma$. 
Moreover, for every $\gamma \in V \cup Y' \cup U' \cup \bigcup_{j=1}^k W_j'$, 
there exists a face $F''$ of $R$ containing all $y_1,\ldots,y_b,u_1,\ldots,u_b$ and $\gamma$. 
Thus, $x$ should be also contained in $F''$. This is a contradiction to what $R$ is simplicial. 
%
%
\end{proof}

As an immediate corollary of Theorem \ref{general}, we obtain the following. 
\begin{Corollary}\label{kei}
Given positive integers $n$ and $\rho$ with $n \geq 5$ and $3 \leq \rho \leq n$, 
there exists an I-isolated smooth Fano $n$-polytope with $n+\rho$ vertices. 
\end{Corollary}
\begin{proof}
If $\rho=3$, then we set $a=n-3$ and $b=2$. By Theorem \ref{pic3cha}, 
there exists an I-isolated smooth Fano $n$-polytope with $n+3$ vertices. Assume $\rho \geq 4$. 
\begin{itemize}
\item If $\rho=4$, then we set $a=n-3, b=1, k=1$ and $l_1=2$. 
\item If $\rho \geq 5$, then we set $a = n - \rho + 2, b=1, k = \rho-3$ and $l_j=1$ for each $1 \leq j \leq k$. 
\end{itemize}
Then, by Theorem \ref{general}, there exists an I-isolated smooth Fano $n$-polytope with $n+\rho$ vertices. 
\end{proof}

\begin{Remark}\label{KN}{\em 
(a) By Corollary \ref{kei}, there are I-isolated smooth Fano 5-polytopes with 8, 9, 10 vertices, respectively. 
On the other hand, we have already seen in Section \ref{5dim} or Figure \ref{figure} that 
there are exactly three I-isolated smooth Fano 5-polytopes. 
However, there are I-isolated smooth Fano 6-polytopes whose primitive relations are not of the forms \eqref{primrel}.

\noindent
(b) Moreover, let $P$ be a convex hull of the following 15 lattice points in $\ZZ^7$: 
\begin{align*}
&\eb_1,\; \eb_2,\; \eb_3,\; \eb_4,\; \eb_5,\; \eb_6,\; \pm\eb_7, \\
&-\eb_1+\eb_7,\; \pm(-\eb_2+\eb_7),\; -\eb_3+\eb_6,\; \eb_2-\eb_6,\; -\eb_2-\eb_4+\eb_6,\; \eb_2-\eb_5-\eb_7. 
\end{align*}
%
%
Then we can see that $P$ is an I-isolated smooth Fano 7-polytope with 15 vertices. 
This is the smallest example of an I-isolated smooth Fano $n$-polytope with at least $2n+1$ vertices. 
}\end{Remark}

\section*{Acknowledgement}
The author would like to thank to Alexander Kasprzyk and Benjamin Nill for their cooperations. 
Thanks to their computation, the author has succeeded in finding the examples in Remark \ref{KN}. 

The author is partially supported by JSPS Fellowship for Young Scientists 
and by the Grant-in-Aid for Young Scientists (B) $\sharp$26800015 from JSPS.


\begin{thebibliography}{1}

\bibitem{Bat81}
V. V. Batyrev, Toric Fano threefolds, 
{\em Izv. Akad. Nauk SSSR Ser. Mat.} {\bf 45} (1981), 704--717. 

\bibitem{Bat91} 
V. V. Batyrev, 
On the classification of smooth projective toric varieties, 
{\em Tohoku Math J.} {\bf 43} (1991), 569--585. 

\bibitem{Bat99}
V. V. Batyrev, On the classification of toric Fano 4-folds, Algebraic geometry, 9. 
{\em J. Math. Sci.} (New York) {\bf 94} (1999), 1021--1050. 

%
\bibitem{Cas03}
C. Casagrande, Contractible classes in toric varieties, 
{\em Math. Z.} {\bf 243} (2003), 99--126. 

%
\bibitem{Ewald}
G. Ewald, 
On the classification of toric Fano varieties, 
{\em Discrete Comp. Geom.} {\bf 3} (1988), 49--54. 

\bibitem{Fulton}
W. Fulton, Introduction to Toric Varieties, 
Ann. of Math. Studies 131, Princeton Univ. Press, 1993. 

%
%
%
\bibitem{Kl88}
P. Kleinschmidt, 
A classification of toric varieties with few generators, 
{\em Aequationes Math.} {\bf 35} (1988), 254--266. 

\bibitem{KN}
M. Kreuzer and B. Nill, Classification of toric Fano 5-folds, 
{\em Adv. Geom.} {\bf 9} (2009), 85--97. 

%

%
%
%
\bibitem{Obro07}
M. {\O}bro, 
An algorithm for the classification of smooth Fano polytopes,
arXiv:0704.0049. 

\bibitem{Obrophd}
M. {\O}bro, 
Classification of smooth Fano polytopes, Ph.D. thesis, University of Aarhus, 2007, 
available at \url{http://pure.au.dk/portal/files/41742384/imf_phd_2008_moe.pdf}

\bibitem{Obro08}
M. {\O}bro, 
Smooth Fano polytopes can not be inductively constructed, 
{\em Tohoku Math J.} {\bf 60} (2008), 219--225. 

%
\bibitem{Sato00}
H. Sato, 
Towards the classification of higher-dimensional toric Fano varieties, 
{\em Tohoku Math J.} {\bf 52} (2000), 383--413. 

%
%
\bibitem{WW}
K. Watanabe and M. Watanabe, The classification of Fano 3-folds with torus embeddings, 
{\em Tokyo J. Math.} {\bf 5} (1982), 37--48. 

\end{thebibliography}
\end{document}